\documentclass[11pt]{article}
\usepackage{amssymb,amsmath,amsthm}
\usepackage[margin=1.0in]{geometry}
\usepackage{tikz}
\usepackage{pgfplots}
\usepackage{pgfplotstable}
\usepackage[hypcap=true]{subcaption}
\usepackage{listings}
\newtheorem{theorem}{Theorem}
\newtheorem{corollary}{Corollary}[theorem]

\author{Robert C.~Kirby \and Daniel Shapero}

\date{\today}
\title{High-order bounds-satisfying approximation of partial differential equations via finite element variational inequalities}

\begin{document}
\maketitle

\begin{abstract}
  Solutions to many important partial differential equations satisfy bounds constraints,
  but approximations computed by finite element or finite difference methods typically fail to respect the same conditions.
  Chang and Nakshatrala~\cite{chang2017variational} enforce such bounds in finite element methods through the solution of variational inequalities rather than linear variational problems.
  Here, we provide a theoretical justification for this method, including higher-order discretizations.
  We prove an abstract best approximation result for the linear variational inequality and estimates showing that bounds-constrained polynomials provide comparable approximation power to standard spaces.
  For any unconstrained approximation to a function, there exists a constrained approximation which is comparable in the $W^{1,p}$ norm.

  In practice, one cannot efficiently represent and manipulate the entire family of bounds-constrained polynomials, but applying bounds constraints to the coefficients of a polynomial in the Bernstein basis guarantees those constraints on the polynomial.
  Although our theoretical results do not guaruntee high accuracy for this subset of bounds-constrained polynomials, numerical results indicate optimal orders of accuracy for smooth solutions and sharp resolution of features in convection-diffusion problems, all subject to bounds constraints.

\end{abstract}

\section{Introduction}
Finite elements provide a powerful, theoretically robust tool kit for the numerical approximation of partial differential equations (PDE).
Frequently, the exact solution of the PDE satisfies not only a variational problem/operator equation, but also some kind of inequality such as a maximum principle.
However, discretizations of the PDE typically do not satisfy these inequalities.
Even for the Poisson equation with piecewise linear elements, a discrete maximum principle is known to be quite delicate and need not hold~\cite{draganescu2005failure}.
The situation is frequently worse with higher-order discretizations or in the presence of more complex operators such as variable coefficients or convective terms.

In many applications, the bounds represent hard constraints, and numerical methods must respect them for the overall computation to give stable and physically relevant results.
In this paper, we build on an approach introduced in~\cite{chang2017variational,layton1996oscillation} for finite element approximations of convection-diffusion problems, in which a linear variational problem is replaced by a variational inequality that finds the closest point to the solution subject to the bounds constraints.
The well-posedness of this approach was established in~\cite{chang2017variational}
and has been applied to porous media problems~\cite{cheng2022scalable, yang2019fully}, and scalable solvers are known~\cite{zhao2023parallel}.

Some theoretical justification of this approach for the convection-diffusion problem with piecewise linear approximations was given in~\cite{layton1996oscillation}.
Here, we give a more general treatment and consider extension to higher orders of approximation.
We give an abstract best approximation theorem, adapting a result of Falk~\cite{falk1974error} for discrete variational inequalities.
Further, extending a result from Despr{\'e}s~\cite{despres2017polynomials}, we show that bounds-constrained approximations can have comparable approximation to the best approximation in $W^{1,p}$.
These theoretical results suggest that solving discrete variational inequalities can produce numerical solutions combining high resolution of features with hard bounds preservation.
Although these results apply to approximation of any order, 
by imposing the bounds constraints on the Bernstein control net, we can guarantee the resulting polynomial will satisfy the bounds constraint.
This does not produce all bounds-satisfying polynomials, and to what degree the theoretical approximating power is reduced is unclear, but our numerical results indicate a major gain over low-order approximations.

We note that the Bernstein basis has also been used to produce higher-order methods for conservation laws via limiting and/or flux correction~\cite{kuzmin2020subcell,lohmann2017flux}.
Our method can be applied in broader contexts and is independent of the particular kind of problem being solved.
Other problem-agnostic methods~\cite{allen2022bounds,zala2022convex} involve some kind of post-processing, such as solving a convex optimization problem, and the technique in~\cite{allen2022bounds} also relies on properties of the Bernstein basis.
We also note the solution of variational inequalities replaces linear problems with more challenging nonlinear ones, but perhaps this is not surprising in light of Godunov's theorem~\cite{godunov1959finite}.
In the more-relevant finite element context, Nochetto and Wahlbin~\cite{nochetto2002positivity} give an impossibility result for linear operators that preserve positivity with high accuracy.
Hence, higher-order bounds-preserving methods must resort to nonlinear processes, such as the variational inequalities we consider here.

In Section~\ref{sec:formulation}, we present the abstract formulation of our problems and give examples of diffusion and convection-diffusion PDEs in this framework.
Then, we give an abstract best approximation result in Section~\ref{sec:approx} and
establish the approximation power of bounds-constrained polynomials in Section~\ref{sec:bounds}, where we also describe the Bernstein basis and sufficient conditions for bounds constraints in that basis.
Finally, we give numerical results in Section~\ref{sec:results} and concluding remarks in Section~\ref{sec:conc}.

\section{Problem formulation}
\label{sec:formulation}
\subsection{Abstract setting}
Let $V$ be a Hilbert space.  Let $a:V \times V \rightarrow \mathbb{R}$ be a continuous and coercive bilinear form.  That is, there exists constants $\alpha, C > 0$ such that
\begin{equation}
  \label{eq:cont}
  a(u, v) \leq C \left\| u \right\| \left\| v \right\|
\end{equation}
for all $u, v \in V$ and
\begin{equation}
  \label{eq:coercive}
  a(u, u) \geq \alpha \left\| u \right\|^2
\end{equation}
for all $u \in V$.

We let $F \in V^\prime$ be a bounded linear functional and consider the variational problem of finding $u \in V$ such that
\begin{equation}
  \label{eq:vareq}
  a(u, v) = F(v) 
\end{equation}
for all $v \in V$.
Existence, uniqueness, and stability are known under the given assumptions thanks to the Lax-Milgram Lemma.

For many PDEs, solutions satisfy additional properties, such as nonnegativity or maximum principles.
To this end, we suppose that $U \subseteq V$ is a closed and convex subset of $V$ such that the solution $u$ to \eqref{eq:vareq} is known to lie in $U$.
However, this inclusion is not implied directly by the variational problem, and need not be respected under discretization.

Following~\cite{chang2017variational}, we can formulate a discrete problem that approximates the solution to~\eqref{eq:vareq} while respecting the bounds constraints as a variational \emph{inequality} rather than linear variational problem.

We take $V_h \subseteq V$ to be some finite-dimensional subspace of $V$, and $U_h \subseteq V_h$ a closed, convex subset.
For our purposes, we also suppose that $U_h \subseteq U$ as well.

Now, we pose the variational inequality of finding $u_h \in U_h$ such that 
\begin{equation}
  \label{eq:varineq}
  a(u_h, v_h - u_h) \geq F(v_h - u_h)
\end{equation}
for all $v_h \in U_h$.
Standard theory~\cite{ciarlet2002finite} implies a unique solution to the discrete problem.

We may characterize approaches to enforcing bounds constraints for discretized PDEs according to two criteria.
First, some approaches are monolithic, integrated into the method itself, while other approaches rely on post-processing the solution.  
Second, we may classify approaches as to whether they are generic or specific to particular classes of problem.
Our technique would be classified as monolithic and generic -- we solve a discrete problem that guarantees constraint satisfaction and can be applied whenever there is an underlying variational problem.
In contrast, techniques such as~\cite{ern2022invariant,kuzmin2020subcell,lohmann2017flux} incorporate the bounds constraints into the discretization, but seem to rely on the structure of particular classes of problems (e.g. conservation laws).
Alternatively, work such as~\cite{allen2022bounds, zala2022convex} gives generic tools independent of the problem structure.
Once one has executed some numerical algorithm, one can then postprocess that solution via some optimization problem to produce a nearby solution satisfying bounds constraints.

\emph{Time-dependent} problems also present challenges in bounds preservation.
The setting we develop here for steady problems can be applied to simple approaches to time stepping.
Now, we consider a Hilbert space $H$ endowed with inner product $(\cdot, \cdot)$ with $V$ compactly embedded in $H$ in turn compactly embedded in $V^\prime$.  In many applications, $V$ will be a subspace of $H^1$ and $H$ will be $L^2$.

Fix $T > 0$ and consider $F(t): [0, T] \rightarrow V^\prime$ representing a time-dependent linear functional.
Then, we consider the evolution equation seeking
$u : (0, T] \rightarrow V$ such that 
\begin{equation}
    \left( u_t , v \right) + a\left( u, v \right) = F(t; v),
\end{equation}
for all $v \in V$.  We impose the initial condition
\begin{equation}
  \label{eq:evolutionic}
  u(\cdot, 0) = u_0 \in V.
\end{equation}

We first discretize in time using some $A$-stable implicit scheme to arrive a sequence of variational problems on $V$.
For example, if we partition $[0, T]$ into some $N$ intervals of size $\tau = T / N$ and define $t_i = i \tau$, we seek a sequence
$\{ u_i \}_{i=1}^N \subset V$ such that:
\begin{equation}
  \left( \frac{u_i - u_{i-1}}{\tau} , v \right)
  + a \left( u_i , v \right) = F(t_i; v).
\end{equation}
Of course, each $u_i$ can be computed from $u_{i-1}$, and $u_0$ is given as an initial condition.
Each variational problem may be rewritten as
\begin{equation}
  \label{eq:vpforbe}
  b(u_i, v) = G_i(v),
\end{equation}
where
\begin{equation}
  \begin{split}
    b(u, v) & = \left( u , v \right) + \tau  a\left( u, v \right) \\
    G_i(v) & = \left( u_{i-1}, v \right) +  \tau  F(t_i; v),
  \end{split}  
\end{equation}

Assuming $a$ is bounded and coercive, the variational problem for each time step will be as well (with constants depending on $\tau$).
We obtain a traditional fully discrete method by replacing $u_0$ with some $u_{h, 0} \in V_h$ and then approximating~\eqref{eq:vpforbe} with a finite element approximation $u_{h, i} \in V_h$ such that
\begin{equation}
  b(u_{h, i}, v_h) = G_i(v_h), \ \ \ v_h \in V_h.
\end{equation}

Alternatively, we can enforce bounds constraints at each time by replacing this discrete variational problem with an inequality as above.
If the true solution of the evolution equation is known to lie in some closed convex set $U$, we can approximate this with $U_h$ as above and then define $u_{h, i} \in U_h$ by the solution of the discrete variational inquality
\begin{equation}
  b(u_{h, i}, v_h - u_{h, i}) \geq G_i(v_h - u_{h, i}).
\end{equation}

We regard time-dependence as an application of our techniques and do not delve deeply into formulating general bounds-constrained methods or error estimates.
The latter topic we anticipate to be accessible using the results we develop here for stationary problems combined with standard techniques like Gr\"onwall-type estimates.

\subsection{Examples}
Finite element methods already exhibit the key issue of constraint violation for relatively simple problems.
Consider an elliptic equation
\begin{equation}
  \label{eq:ellipticPDE}
  - \nabla \cdot \left( \kappa \nabla u \right) = f,
\end{equation}
posed over a domain $\Omega \subset \mathbb{R}^d$ with $d=1, 2, 3$.
The coefficient $\kappa$ is generically a mapping from $\Omega$ into symmetric tensors, uniformly positive-definite over $\Omega$.

Partitioning the boundary $\partial \Omega$ into $\Gamma^D \cap \Gamma^N$, we impose boundary conditions
\begin{equation}
  \begin{split}
    u & = g, \ \ \ \Gamma^D, \\
    - \kappa \nabla u \cdot n & = \gamma \ \ \ \Gamma^N,
  \end{split}
\end{equation}
where $n$ is the unit normal outward to $\Omega$.

Much is known about maximum principles for this problem~\cite{evans2022partial, mudunuru2016enforcing}.
In particular, for positive $f$ and $\gamma=0$ or $\Gamma^N=\emptyset$, the solution $u$ will be bounded below by the minimum of $g$.

To arrive at a variational form of this equation, we let $H^1(\Omega)$ be the standard Sobolev space of square-integrable functions with square-integrable first-order weak derivatives.
Let $V$ be the subspace of $H^1_0(\Omega)$ consisting of functions whose trace vanishes on $\Gamma^D$.
Then, we seek find $u \in H^1(\Omega)$ with trace equal to $g$ on $\Gamma^D$ such that
\begin{equation}
  a(u, v) = F(v), \ \ \ v \in V,
\end{equation}
where
\begin{equation}
  \label{eq:diffvar}
  \begin{split}
    a(u, v) & = \int_\Omega \kappa \nabla u \cdot \nabla v \, dx, \\ 
    F(v) & = \int_\Omega f v \, dx + \int_{\Gamma^n} \gamma v \, ds.
  \end{split}
\end{equation}

We introduce a triangulation $\mathcal{T}_h$ of $\Omega$~\cite{brenner2008mathematical} and let $V_h$ consist of continuous piecewise polynomials of some degree $k$ over that triangulation.
Then, a standard Galerkin method follows by restricting the variational problem to the finite element subspace, producing $u_h \in V_h$ such that
\begin{equation}
a(u_h, v_h) = F(v_h), \ \ \ v_h \in V_h.
\end{equation}

In order to define a discrete variational inequality, we need to specify the set $U_h$ satisfying the bounds constraints and then solve~\eqref{eq:varineq} over that set with the specified $a$ and $F$ from~\eqref{eq:diffvar}.
For continuous piecewise linear elements, defining $U_h$ is straightforward -- it consists of members of $V_h$ whose nodal values satisfy the bounds constraints.
We are interested, however, in using higher-order approximations, and discuss construction of $U_h$ in this case in the sequel.

Bounds constraints are even more relevant to convection-diffusion problems, where methods frequently include at least mild oscillations around fronts.
We consider the equation
\begin{equation}
  \label{eq:convdiff}
  - \nabla \cdot \left( \kappa \nabla u \right) + \beta \cdot \nabla u = f,
\end{equation}
where $\kappa$ is as before and $\beta$ may vary spatially but is assumed divergence-free.

As with the diffusive equation, we partition $\partial \Omega$ into $\Gamma^D \cap \Gamma^N$.
The Neumann boundary $\Gamma^N$ is further partitioned into outflow and inflow portions $\Gamma^N_{in}$ and $\Gamma^N_{out}$ with
$\beta \cdot n$ nonnegative or negative, respectively.
We pose the boundary conditions
\begin{equation}
  \begin{split}
    u & = g, \ \ \Gamma^D, \\
    n \cdot \left( \beta u - \kappa \nabla u \right) & = \gamma, \ \ \Gamma^N_{in}, \\
    -n \cdot \left( \kappa \nabla u \right) & = \gamma, \ \ \Gamma^N_{out}.
  \end{split}
\end{equation}
A Galerkin method follows from
\begin{equation}
  a(u, v) = F(v),
\end{equation}
where
\begin{equation}
  \label{eq:cgcda}
  a(u, v) = \int_\Omega \kappa \nabla u \cdot \nabla v  + \left( \beta \cdot \nabla u\right) v \, dx
\end{equation}
\begin{equation}
  F(v) = \int_\Omega f v \, dx - \int_{\Gamma^N} \gamma v \, ds
\end{equation}
However, when $\beta$ is large relative to $\kappa$, the method loses stability unless a very fine mesh is used.
Among the many choices available to resolve this difficulty, we take the classical SUPG method~\cite{brooks1982streamline}.
This method gives rise to the variational problem
\begin{equation}
  \label{eq:supg}
  a_{SUPG}(u, v) = F_{SUPG}(v),
\end{equation}
where
\begin{equation}
  a_{SUPG} = a(u, v) + \int_\Omega \delta \left( \beta \cdot \nabla u - \nabla \cdot \kappa \nabla u \right) \left( \beta \cdot \nabla v \right) \, dx
\end{equation}
and
\begin{equation}
  F_{SUPG}(v) = F(v) + \int_\Omega \delta f \left( \beta \cdot \nabla v \right)\, dx.
\end{equation}
Here, $\delta$ is some parameter controlling the size of stabilizing term, and we take
\begin{equation}
  \delta = \tfrac{h}{2\|\beta\|}.
\end{equation}

Finally, we also consider time-dependent convection diffusion,
\begin{equation}
  \label{eq:tdcd}
  u_t + \beta \cdot \nabla u - \nabla \cdot \left( \kappa \nabla u \right),
\end{equation}
with boundary conditions as considered above in the steady case.
It turns out that implicit time stepping plus variational inequalities give good stability results, so consider standard Galerkin rather than SUPG discretizations, taking the evolution equation
\begin{equation}
  \label{eq:tdcdgal}
  \left( u_t , v \right) + a(u, v) = F(t; v)
\end{equation}
as our model, $a$ as given in~\eqref{eq:cgcda}.

\section{Best approximation result}
\label{sec:approx}
Here, we provide a best approximation result for $u - u_h$.
Such a result could follow as a special case of Falk's treatment of variational inequalities~\cite{falk1974error}.
In our setting, however, $u$ satisfies a variational equation rather than an inequality.
Consequently, we do not require the introduction of a pivot space and associated extra regularity assumptions in Falk's treatment.
The result is also further simplified by requiring $U_h \subset U$.
\begin{theorem}
  \label{thm:ba}
  Let $V$ be a Hilbert space with closed, convex subset $U$.  Let
  bilinear form $a$ satisfy~\eqref{eq:cont} and~\eqref{eq:coercive} and $F \in V^\prime$.
  Let $u \in V$ satisfy the variational problem~\eqref{eq:vareq} as well as $u \in U$.  
  For any finite-dimensional subspace $V_h \subseteq V$ with closed convex subset $U_h \subseteq V_h \cap U$, the solution $u_h \in U_h$ to~\eqref{eq:varineq} satisfies the best approximation result
  \begin{equation}
    \left\| u - u_h \right\| \leq \tfrac{C}{\alpha} \inf_{v_h \in U_h} \left\| u - v_h \right\|.
  \end{equation}
\end{theorem}
\begin{proof}
  Using coercivity, we have
  \begin{equation*}
    \begin{split}
      \alpha \left\| u - u_h \right\|^2 & \leq a(u-u_h, u-u_h)  \\
      & = a(u, u) + a(u_h, u_h) - a(u, u_h) - a(u_h, u).
    \end{split}
  \end{equation*}
  Now,~\eqref{eq:varineq} implies that for any $v_h \in U_h$,
  \[
  a(u_h, u_h) \leq a(u_h, v_h) + F(u_h - v_h),
  \]
  so
  \begin{equation*}
    \alpha \left\| u - u_h \right\|^2
    \leq a(u, u) - a(u, u_h) - a(u_h, u) + a(u_h, v_h) + F(u_h - v_h).
  \end{equation*}
  Because $u$ satisfies the variational equality~\eqref{eq:vareq}, we can write
  \[
  F(u_h - v_h) = F(u_h) - F(v_h) = a(u, u_h) - a(u, v_h),
  \]
  and so
  \begin{equation*}
    \begin{split}
    \alpha \left\| u - u_h \right\|^2
    & \leq a(u, u) - a(u_h, u) + a(u_h, v_h) - a(u, v_h) \\
    & = a(u - u_h, u) - a(u - u_h, v_h) \\
    & = a(u - u_h, u - v_h),
    \end{split}
  \end{equation*}
  and the result follows from the continuity estimate~\eqref{eq:cont}.
\end{proof}

This result is an analog of C\'ea's Lemma~\cite{cea1964approximation} for variational problems.
Hence, we expect error estimates in specific settings to be determined by the regularity of the solution $u$ to~\eqref{eq:vareq} and the approximation power of of the solution set $U_h$.

Although a variational inequality produces the numerical approximation,
the regularity of $u$ will be determined by standard PDE theory.
Hence, the appearance of a variational inequality does not artificially restrict the approximability, and if $u$ is smooth enough the accuracy will only be limited by how well it is approximated in $U_h$.

At this point, our convergence theory covers only consistent and conforming discretizations with a coercive operator.
Variational inequalities have been formulated for other methods, such as discontinuous Galerkin and mixed methods, in~\cite{chang2017variational}, although we leave the analysis of more general settings to future work.

\section{Constrained approximation}
\label{sec:bounds}
\subsection{Approximation by bounds-constrained polynomials}
Best approximation by bounds-respecting polynomials is a challenging problem that has received relatively little attention in the literature.
An early result here is due to Beatson~\cite{beatson1982restricted}.
Working with maximal-continuity univariate splines with possibly nonuniform knot spacing, his main result is the existence of optimal-order nonnegative spline approximations to smooth nonnegative functions.
Estimates are given for all derivatives in $L^p$ norms for $1 \leq p \leq \infty$.

More recently, a general result for univariate polynomials on an interval was given by Despr{\'e}s~\cite{despres2017polynomials}.
He showed that for any polynomial approximation $q$ of some continuous $f:[0, 1]\rightarrow[0,1]$ there exists a polynomial $g$ of the same degree satisfying the bounds constraints with $L^\infty$ approximation error no more than twice that of the error in $g$. 

Here, we give a major generalization of his result, which essentially recognizes that the assumptions can be weakened and adapts his construction to functions with general bounded range rather than $[0, 1]$.  
\begin{theorem}
  \label{thm:boundapprox}
  Let $K \subset \mathbb{R}^n$ be a compact domain and $f: K \rightarrow \mathbb{R}$ be a continuous function with range $[m, M]$.
  Let $V_h$ be a vector space of continuous functions mapping $K$ into $\mathbb{R}$ such that constant-valued functions on $K$ are included.
  For any $g \in V_h$, there exists $q \in V_h$ such that the range of $q$ is contained in $[m, M]$ and
  \begin{equation}
    \| f - q \|_\infty \leq 2 \| f - g \|_\infty.
  \end{equation}
\end{theorem}
\begin{proof}
  We define $\overline{m} = \tfrac{m+M}{2}$.
  Let $g \in V_h$ be given and put
  \begin{equation}
    \label{eq:defofq}
    q = \overline{m} + \tfrac{M-m}{M-m+2\|f-g\|_\infty} \left( g - \overline{m} \right).
  \end{equation}
  Since $q$ is constructed by vector space operations and $V_h$ contains constants, we have $q \in V_h$.

  We note
  \[
  \| f - \overline{m} \|_{\infty} \leq \tfrac{M-m}{2}.
  \]
  Then, we also observe
  \[
  \| g - \overline{m} \|_\infty \leq \| f - \overline{m} \|_\infty + \| f - g \|_\infty \leq \tfrac{M-m}{2} + \| f - g \|_\infty.
  \]
  This gives that
  \[
  \begin{split}
  \| q - \overline{m} \|_\infty & =
  \tfrac{M-m}{M-m+2\|f-g\|_\infty} \| g - \overline{m} \|_\infty \\
  & \leq
  \tfrac{M-m}{M-m+2\|f-g\|_\infty} \left( \tfrac{M-m}{2} + \| f - g \|_\infty \right) \\
  & = \tfrac{M-m}{M-m+2\|f-g\|_\infty} \left( \tfrac{M-m + 2 \| f - g \|_\infty}{2} \right) \\
  & = \tfrac{M-m}{2},
  \end{split}
  \]
  and so the range of $q$ is contained in $[m, M]$.

  Now, the difference between $f$ and $q$ is given by
  \begin{equation}
    \begin{split}
      f - q & = f - \overline{m} - \tfrac{M-m}{M-m+2\|f-g\|_\infty} \left( g - \overline{m} \right) \\
       & = \left( 1 - \tfrac{M-m}{M-m+2\|f-g\|_\infty} \right) \left( f - \overline{m} \right) + \tfrac{M-m}{M-m+2\|f-g\|_\infty} \left( f - g \right),
    \end{split}
  \end{equation}
  and so
  \begin{equation}
    \label{eq:begindiff}
    \begin{split}
      f - q & = \left( 1 - \tfrac{M-m}{M-m+2\|f-g\|_\infty} \right) \left( f - \overline{m} \right) + \tfrac{M-m}{M-m+2\|f-g\|_\infty} \left( f - g \right) \\
       & = \left( \tfrac{2 \| f - g \|_\infty}{M-m+2\|f-g\|_\infty} \right)\left( f - \overline{m} \right) + \tfrac{M-m}{M-m+2\|f-g\|_\infty} \left( f - g \right).
    \end{split}
  \end{equation}
  Taking norms of both sides, we have
  \begin{equation}
    \begin{split}
      \| f - q\|_\infty & \leq \left( \tfrac{2 \| f - g \|_\infty}{M-m+2\|f-g\|_\infty} \right) \tfrac{M-m}{2} + \tfrac{M-m}{M-m+2\|f-g\|_\infty} \| f - g \|_\infty \\
      & = \tfrac{2 (M-m) \| f -g \|_\infty}{M-m + 2 \| f - g \|_\infty} 
      = \tfrac{2 \| f - g \|_\infty}{1+\tfrac{2}{M-m} \|f-g \|_\infty} \\
      & \leq 2 \| f - g \|_\infty.
    \end{split}
  \end{equation}
\end{proof}

This theorem applies not only to spaces of polynomials, but also piecewise polynomials.
However, if the function space $V_h$ is assumed to satisfy boundary conditions, then the construction of $q$ given in~\eqref{eq:defofq} need not lie in $V_h$.
Hence, our result does not directly apply to Theorem~\ref{thm:ba} in such cases, although it does if all boundary conditions are included via the variational form rather than strongly through the definition of function spaces.
Our result can also be applied piecewise to the case of approximation by piecewise polynomials lacking inter-element continuity.
It also works on nonpolynomial approximations such as exponential or trigonometric functions provided that constants are included in the approximating space.

The construction in Theorem~\ref{thm:boundapprox} can easily be integrated over $K$ to give estimates in $L^p$ norms.
\begin{corollary}
  \label{lp}
  Let $1 \leq p < \infty$.  
  Under the assumptions of Theorem~\ref{thm:boundapprox}, there exists $C > 0$ depending on $K$ and $p$ but not $f$ such that for any $g \in V_h$ there exists $q \in V_h$ with range contained in $[m, M]$ and
  \begin{equation}
    \| f - q \|_p \leq C \| f - g \|_p.
  \end{equation}
\end{corollary}

Theorem~\ref{thm:boundapprox} is abstract, but we primarily envision its application to piecewise polynomial finite element methods.
This context gives us a clean pathway to estimates of derivatives via inverse estimates~\cite{brenner2008mathematical}.
To this end, we suppose that $V_h$ is constructed by continuous (or smoother) piecewise polynomials over a triangulation of $K$ and that functions in the space satisfy an inverse estimate.
That is, we suppose that there exists some $C_I$ such that
\begin{equation}
  \label{eq:inverse}
  \| \nabla q \|_p \leq \tfrac{C_I}{h} \| q \|_p,
\end{equation}
where $h$ is the maximum diameter over cells in the triangulation.  This
is consistent with the inverse estimates proven for standard finite elements in~\cite{brenner2008mathematical}.

\begin{theorem}
  \label{thm:derivbound}
  Under the assumptions of Theorem~\ref{thm:boundapprox} and~\eqref{eq:inverse}, there exists $C>0$ such that for any $g \in V_h$, there exists some $q \in V_h$ with range contained in $[m, M]$ such that
  \begin{equation}
  \| \nabla( f - q ) \|_p \leq \| \nabla ( f - g ) \|_p + \tfrac{C}{h} \| f - g \|_p.
  \end{equation}
\end{theorem}
\begin{proof}
  The inverse estimate~\eqref{eq:inverse} between applications of the triangle inequality gives
  \[
  \begin{split}
    \| \nabla( f - q ) \|_p &
    \leq \| \nabla( f - g ) \|_p + \| \nabla( g - q ) \|_p \\
    & \leq \| \nabla( f - g ) \|_p + \tfrac{C_I}{h} \| g - q \|_p \\
    & \leq \| \nabla( f - g ) \|_p + \tfrac{C_I}{h} \left( \| f - g \|_p + \| f - q \|_p \right),
  \end{split}
  \]
  and Corollary~\ref{lp} or Theorem~\ref{thm:boundapprox} finishes the estimate.
\end{proof}
Together, these results imply $W^{1,p}$ estimates, and higher-order derivatives may be estimated by similar techniques provided sufficient smoothness of $f$ and the approximating space.
As with our discussion of Theorem~\ref{thm:boundapprox}, this result may also be applied piecewise for discontinuous approximating spaces.

\subsection{Representing bounds-constrained polynomials}
This theoretical discussion shows that enforcing bounds constraints on the numerical solution of PDEs via variational inequalities can produce errors comparable to the unconstrained solution and hence need not degrade the overall accuracy of the approximation.
While the numerical results in~\cite{chang2017variational} and works based on it use lowest-order (piecewise linear) basis approximations, these do not employ higher-order approximations.

While continuous piecewise linear polynomials satisfy bounds constraints iff their nodal values do, encoding the bounds constraints for higher-order and/or multivariate polynomials is quite challenging.
In one variable, Nesterov~\cite{nesterov2000squared} has classified nonnegative polynomials over an interval in terms of convex cones of their coefficients, with the coefficients lying in this cone iff the polynomial admits a certain sum-of-squares representation.
However, the general problem of determining nonnegativity of a multivariate polynomial is actually NP-hard~\cite{lasserre2007sum}.

Based on our prior work in~\cite{allen2022bounds}, we consider sufficient conditions for bounds constraints based on the \emph{Bernstein} basis~\cite{bernstein1912demo, LaiSch07}.
In one variable on $[0, 1]$, these polynomials of degree $n$  take the form
\begin{equation}
  b^n_i(x) = {n \choose i} x^i (1-x)^{n-i}, \ \ \ 0 \leq i \leq n.
\end{equation}
These polynomials give a nonnegative partition of unity forming a basis for polynomials of degree $n$.  They are readily mapped to any compact interval $[a, b]$, and given their geometric decomposition~\cite{arnold2009geometric}, they can be easily assembled across cells to form $C^0$ piecewise polynomials.  The cubic basis is given in Figure~\ref{fig:bern3}.

\begin{figure}
  \begin{center}
    \begin{tikzpicture}[scale=3.0]
      \draw[->] (0, 0) -- (1.1, 0) node[right] {$x$};
      \draw[->] (0, 0) -- (0, 1.1) node[above] {$y$};
      \draw [smooth, samples=100, domain=0:1] plot({\x}, {(1-\x)^3});
      \draw [smooth, samples=100, domain=0:1] plot({\x}, {3*\x * (1-\x)^2});
      \draw [smooth, samples=100, domain=0:1] plot({\x}, {3*(\x)^2 * (1-\x)});
      \draw [smooth, samples=100, domain=0:1] plot({\x}, {\x^3});
  \end{tikzpicture}
  \end{center}
  \caption{Cubic Bernstein polynomials.}
  \label{fig:bern3}
\end{figure}
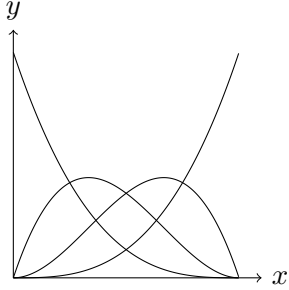

The Bernstein polynomials all take values between 0 and 1, so if the coefficients of some
\[
p(x) = \sum_{i=0}^n p_i B^n_i(x)
\]
take values in the range $[c, d]$, then so must $p(x)$.
More precisely, that $p$ must lie in the convex hull of its control net.

In particular, nonnegative Bernstein coefficients imply nonnegativity of a polynomial, but the converse does not hold.  For example, we can take
\begin{equation}
  \label{eq:oops}
  p(x) = B^2_0(x) - 0.9 B^2_1(x) + B^2_2(x) = 3.8x^2 - 3.8 x + 1,
\end{equation}
which has a negative coefficient but minimum value of 0.05 at $x=0.5$.
It is possible to precisely give (nonlinear!) conditions on the Bernstein coefficients for which a quadratic polynomial is nonnegative~\cite{LaiSch07}.
Necessary and sufficient conditions for general order are given by a theorem of Bernstein that a polynomial $p$ of degree $n$ is positive iff there exists a degree $m$ such that its representation in the Bernstein basis of degree $n+m$ has positive coefficients.
However, this degree $m$ is not known \emph{a priori} and can be large in pathological cases.

The Bernstein basis extends readily to simplices in higher dimensions via barycentric coordinates.
Let $\{ \mathbf{v}_i \}_{i=0}^d$ be the vertices of some nondegenerate simplex $T$ in $\mathbb{R}^d$.
Associated with these, we define the barycentric coordinates $\{ b_i \}_{i=0}^d$.
Each $b_i$ is an affine map from $\mathbb{R}^d$ to $\mathbb{R}$ satisfying
\(
b_i(\mathbf{v}_j) = \delta_{ij}
\).
The barycentric coordinates are nonnegative over $T$ and sum to 1 at all points in $\mathbb{R}^d$.
Using multiindex notation, we let $\boldsymbol{\alpha} = \left( \alpha_0, \dots, \alpha_d \right)$ be a $(d+1)$-tuple of nonnegative integers.
Its order, denoted by $|\boldsymbol{\alpha}|$, is given by
\[
|\boldsymbol{\alpha}| = \sum_{i=0}^d \alpha_i.
\]
We define its factorial as the product of the factorials of its components:
\[
\boldsymbol{\alpha}! = \prod_{i=0}^d \alpha_i!.
\]
For degree $n$ and any $|\boldsymbol{\alpha}| = n$, we define
\begin{equation}
  B^n_{\boldsymbol{\alpha}} = \frac{n!}{\boldsymbol{\alpha}!} \prod_{i=0}^d b_i^{\alpha_i},
\end{equation}
and then the set $\left\{ B^n_{\boldsymbol{\alpha}} \right\}_{i=0}^d$ forms a basis for polynomials of total degree $n$ over $T$.

The multivariate simplicial Bernstein basis retains many favorable properties from the univariate case -- it is nonnegative over $T$, geometrically decomposed so readily assembled in a $C^0$ fashion, and forms a partition of unity.
Moreover, it is also nonnegative over $T$ so that our geometric discussion holds.

Although the multivariate basis possesses structure leading to fast finite element algorithms~\cite{ainsworth2011bernstein, kirby2011fast, kirby2012fast}, our main interest with the Bernstein basis is geometric.
Suppose we wish to approximate a PDE~\eqref{eq:vareq} using the
variational inequality~\eqref{eq:varineq} where the admissible set $U$ consists of nonnegative functions.
As with a standard finite element method, we take $V$ to comprise $H^1$ functions suitably constrained on the boundary, and then we take the standard approximating space with polynomials of degree $k$:
\begin{equation}
  V_h = \left\{ u \in V: u|_{T} \in P^k(T), \ \ T \in \mathcal{T}_h \right\}.
\end{equation}
Now, owing to the difficulties with representing nonnegative polynomials, we cannot simply take $U_h = V_h \cap U$, but we can take $U_h \subset V_h \cap U$ to consist of all members of $V_h$ whose Bernstein coefficients are nonnegative over each cell $T$ in the mesh.

We will see in the sequel that such approximations can indeed increase the accuracy relative to continuous piecewise polynomials, but the question of approximation power of our choice of $U_h$ is an open and rather delicate question.
Indeed, the example~\eqref{eq:oops} shows that $U_h$ can cannot reproduce all positive quadratic functions, so we cannot appeal to standard tools such as the Bramble-Hilbert Lemma~\cite{bramble1971bounds}.
On the other hand, subdividing the interval $[0, 1]$ into two intervals immediately removes the problem, as the Bernstein coefficients on each half-interval would all be positive.

\section{Numerical results}
\label{sec:results}
To illustrate the theory above, we consider a suite of numerical experiments, solving both standard finite element methods and linear variational inequalities for equations~\eqref{eq:ellipticPDE} and~\eqref{eq:convdiff}.
All of our numerical experiments are carried out using Firedrake~\cite{FiredrakeUserManual, Rathgeber:2016}, a high-level Python library that automates the solution of the finite element method.
Firedrake provides convenient high-level access to PETSc~\cite{petsc-user-ref, petsc-efficient} through petsc4py~\cite{Dalcin2011}.
All linear variational problems were all solved with the default PETSc sparse direct LU factorization, and the variational inequalities were solved with PETSc's reduced space active set solver for variational inequalities, \lstinline{vinewtonrsls}, with an absolute stopping tolerance of $10^{-8}$.
For the diffusion equation~\eqref{eq:ellipticPDE}, we meshed the unit square using Firedrake's built-in function, but our meshes for~\eqref{eq:convdiff} were generated using Firedrake's run-time interface to the NETGEN~\cite{schoberl1997netgen} Python bindings.

\subsection{Diffusion}
For purely diffusive problems of the form~\eqref{eq:ellipticPDE}, we perform two kinds of experiments.
First, via the method of manufactured solutions, we empirically assess the accuracy of our method for enforcing bounds constraints.
To this end, we select $\Omega$ to be the unit square, subdividing it into an $N \times N$ mesh of squares, each subdivided into right triangles.
We choose the diffusion coefficient $\kappa$ as in~\cite{chang2017variational}
\begin{equation}
  \label{eq:kappa}
  \kappa = \begin{bmatrix}
    y^2 + \epsilon x^2 & -(1-\epsilon) x y \\
    - (1-\epsilon) x y & x^2 + \epsilon y^2 \end{bmatrix},
\end{equation}
where we pick $\epsilon = 10^{-4}$.
Hence, $\kappa$ is symmetric and positive definite away from the boundary but anisotropic and heterogeneous. 

With homogeneous Dirichlet boundary conditions, we select the forcing function $f$ such that the exact solution of~\eqref{eq:ellipticPDE} is
\[
u(x, y) = e^{2xy} \left( \sin^2 \pi x \right)\left( \sin^2 2\pi y \right).
\]
We solved this problem on an $N \times N$ mesh with $N=2^i$ for $2 \leq i \leq 6$ and polynomials of degree $1 \leq k \leq 3$ using both a linear variational problem and a discrete variational inequality to enforce bounds constraints.

In Figure~\ref{fig:mmsdiff}, we compare the accuracy obtained by our methods in the $L^2$ (Figure~\ref{mmsl2}) and $H^1$ (Figure~\ref{mmsh1}) norms.
The solution of the variational inequality with piecewise linear approximations actually gave slightly smaller $L^2$ and $H^1$ errors than the solution to the discrete variational problem.
This is perhaps surprising, but does not contradict the convergence theory.
The solution of the (symmetric!) variational problem is only known to be optimal in the energy norm:
\[
\| u \|^2_a = \int_\Omega \kappa \nabla u \cdot \nabla u \, dx.
\]
We in fact found the error of the discrete variational inequality slightly larger than that of the variational problem in this norm, although Figure~\ref{mmsen} shows the results are in fact very similar.

\begin{figure}
  \centering
  \begin{subfigure}[t]{0.3\textwidth}
    \begin{tikzpicture}[scale=0.6]
      \begin{axis}[
          legend cell align=left,
          legend pos=south west,
          xlabel={$N$},
          ylabel={$L^2$ error},
          ymode=log, log basis y={10},
          xmode=log, log basis x={2}
        ]
        \addplot table[x=N, y=ul2, col sep= comma] {diffmms_deg1.csv};
        \addlegendentry{$k=1$, VP}
        \addplot table[x=N, y=cl2, col sep= comma] {diffmms_deg1.csv};
        \addlegendentry{$k=1$, VI}
        \addplot table[x=N, y=ul2, col sep= comma] {diffmms_deg2.csv};
        \addlegendentry{$k=2$, VP}
        \addplot table[x=N, y=cl2, col sep= comma] {diffmms_deg2.csv};
        \addlegendentry{$k=2$, VI}
        \addplot table[x=N, y=ul2, col sep= comma] {diffmms_deg3.csv};
        \addlegendentry{$k=3$, VP}
        \addplot table[x=N, y=cl2, col sep= comma] {diffmms_deg3.csv};
        \addlegendentry{$k=3$, VI}      
      \end{axis}
    \end{tikzpicture}
    \caption{$L^2$ error}
    \label{mmsl2}
  \end{subfigure}
  \begin{subfigure}[t]{0.3\textwidth}
    \begin{tikzpicture}[scale=0.6]
      \begin{axis}[
          legend cell align=left,
          legend pos=south west,
          xlabel={$N$},
          ylabel={$H^1$ error},
          ymode=log, log basis y={10},
          xmode=log, log basis x={2}
        ]
        \addplot table[x=N, y=uh1, col sep= comma] {diffmms_deg1.csv};
        \addlegendentry{$k=1$, VP}
        \addplot table[x=N, y=ch1, col sep= comma] {diffmms_deg1.csv};
        \addlegendentry{$k=1$, VI}
        \addplot table[x=N, y=uh1, col sep= comma] {diffmms_deg2.csv};
        \addlegendentry{$k=2$, VP}
        \addplot table[x=N, y=ch1, col sep= comma] {diffmms_deg2.csv};
        \addlegendentry{$k=2$, VI}
        \addplot table[x=N, y=uh1, col sep= comma] {diffmms_deg3.csv};
        \addlegendentry{$k=3$, VP}
        \addplot table[x=N, y=ch1, col sep= comma] {diffmms_deg3.csv};
        \addlegendentry{$k=3$, VI}      
      \end{axis}
    \end{tikzpicture}
    \caption{$H^1$ error}    
    \label{mmsh1}    
  \end{subfigure}
  \begin{subfigure}[t]{0.3\textwidth}
    \begin{tikzpicture}[scale=0.55]
      \begin{axis}[
          legend cell align=left,
          legend pos=south west,
          xlabel={$N$},
          ylabel={Energy error},
          ymode=log, log basis y={10},
          xmode=log, log basis x={2}
        ]
        \addplot table[x=N, y=uen, col sep= comma] {diffmms_deg1.csv};
        \addlegendentry{$k=1$, VP}
        \addplot table[x=N, y=cen, col sep= comma] {diffmms_deg1.csv};
        \addlegendentry{$k=1$, VI}
        \addplot table[x=N, y=uen, col sep= comma] {diffmms_deg2.csv};
        \addlegendentry{$k=2$, VP}
        \addplot table[x=N, y=cen, col sep= comma] {diffmms_deg2.csv};
        \addlegendentry{$k=2$, VI}
        \addplot table[x=N, y=uen, col sep= comma] {diffmms_deg3.csv};
        \addlegendentry{$k=3$, VP}
        \addplot table[x=N, y=cen, col sep= comma] {diffmms_deg3.csv};
        \addlegendentry{$k=3$, VI}      
      \end{axis}
    \end{tikzpicture}
  \caption{Energy error}
    \label{mmsen}
  \end{subfigure}
  \caption{Error approximating~\eqref{eq:ellipticPDE} versus $N$ on an $N \times N$ mesh for various polynomial degrees.  We compare the finite element solution obtained by solving a linear variational problem to the solution of a discrete variational inequality enforcing bounds on the Bernstein coefficients.  In each case, we observe similar convergence rates for the solution by each technique.}
  \label{fig:mmsdiff}
\end{figure}
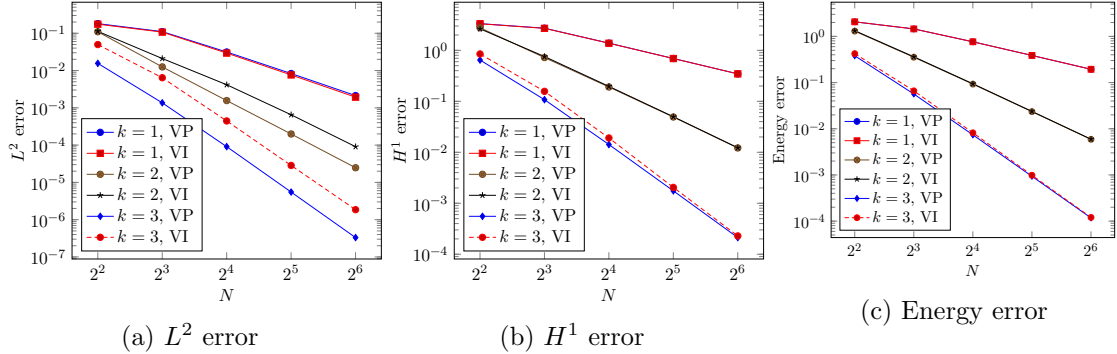

We give contour plots of the solution obtained using quadratic Bernstein basis functions on a $16\times 16$ mesh in Figure~\ref{fig:mmsplot}.
The solution of the linear variational problem in Figure~\ref{mmsvpplot} shows some regions of undershoot in the top left corner and along the line $y=0.5$, where the solution is zero.
In contrast, the solution of the variational inequality shown in Figure~\ref{mmsviplot} shows a uniformly nonnegative solution.
The difference between these two solutions is given in Figure~\ref{mmsdiffplot}. 

\begin{figure}
  \centering
  \begin{subfigure}[t]{0.32\textwidth}
    \includegraphics[width=\textwidth]{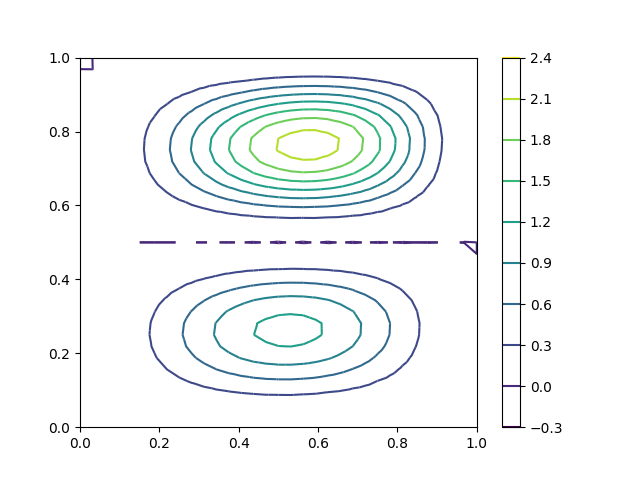}
    \caption{Variational problem}
    \label{mmsvpplot}
  \end{subfigure}
  \begin{subfigure}[t]{0.32\textwidth}
    \includegraphics[width=\textwidth]{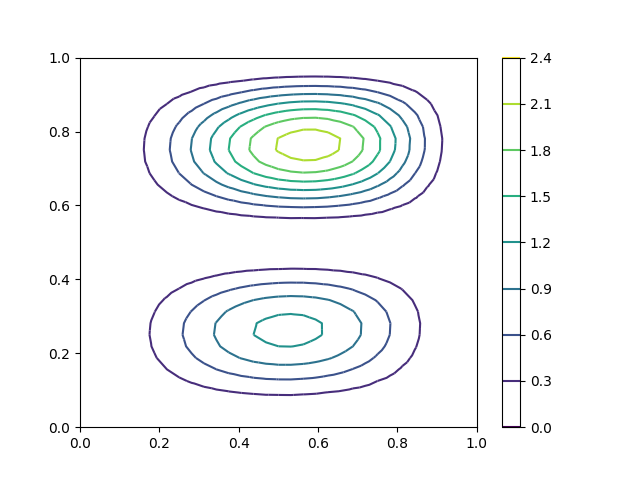}
    \caption{Variational inequality}
    \label{mmsviplot}
  \end{subfigure}
  \begin{subfigure}[t]{0.32\textwidth}
    \includegraphics[width=\textwidth]{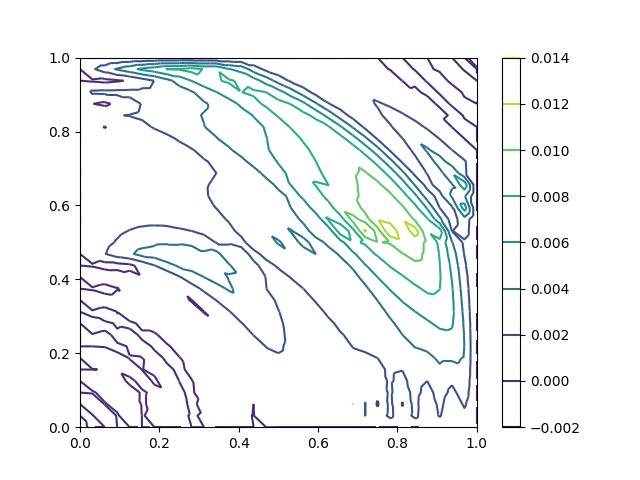}
    \caption{Difference}
    \label{mmsdiffplot}
  \end{subfigure}  
  \caption{Plots of solution to the variational problem, variational inequality, and their difference using quadratic Bernstein polynomials on a $16 \times 16$ mesh.  Notably, the variational inequality gives a uniformly nonnegative solution, unlike the solution of the variational problem.}
  \label{fig:mmsplot}
\end{figure}

Now, we also applied our methods to a problem where the analytic solution is not known.
We keep the same diffusion tensor $\kappa$ as before, but take the forcing function to be
\begin{equation}
  \label{eq:roughf}
  f(x, y) = \begin{cases} 1, & x, y \in [\tfrac{3}{8}, \tfrac{5}{8}]^2, \\
    0, & \text{otherwise}.
  \end{cases}
\end{equation}
This forcing function lies in $H^s$ for $s < \tfrac{1}{2}$ so that the solution is in $H^{2+s}$ for $s<\tfrac{1}{2}$.
Although this regularity does not allow high orders of convergence, the variational problem and inequality are still well-posed with higher-degree polynomials, and we may still hope to obtain better solutions than with piecewise linear polynomials on the same mesh.

\begin{figure}
  \centering
  \begin{subfigure}[t]{0.32\textwidth}
    \includegraphics[width=\textwidth]{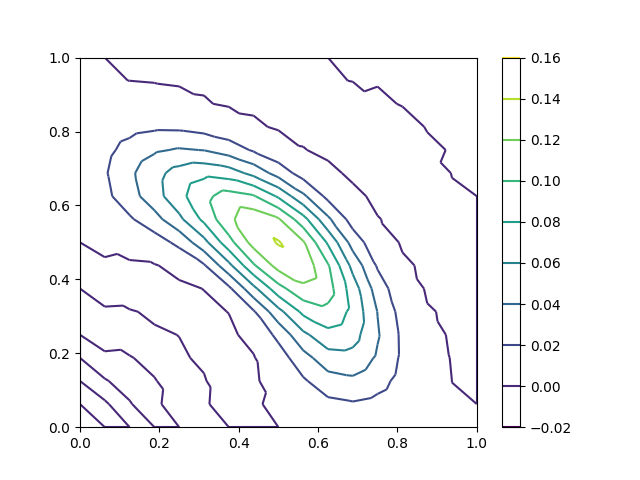}
    \caption{VP, $k=1$}
    \label{l2forcevpplot1}
  \end{subfigure}
  \begin{subfigure}[t]{0.32\textwidth}
    \includegraphics[width=\textwidth]{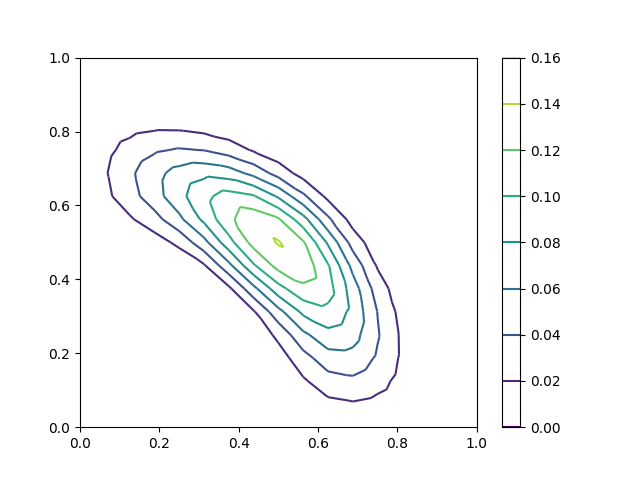}
    \caption{VI, $k=1$}
    \label{l2forceviplot1}
  \end{subfigure}
  \begin{subfigure}[t]{0.32\textwidth}
    \includegraphics[width=\textwidth]{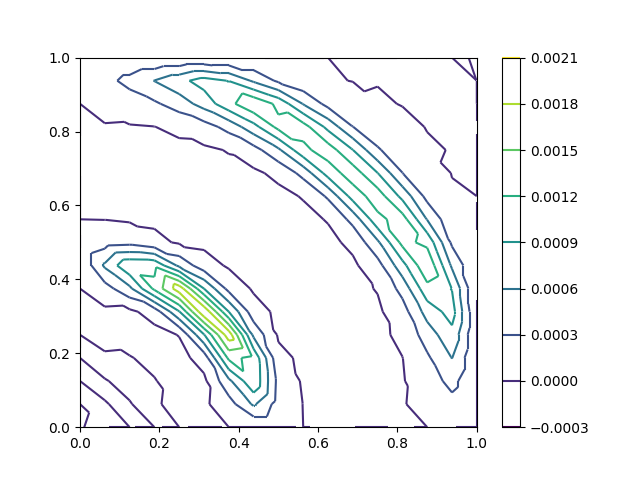}
    \caption{Difference, $k=1$}
    \label{l2forcediffplot1}
  \end{subfigure} \\
  \begin{subfigure}[t]{0.32\textwidth}
    \includegraphics[width=\textwidth]{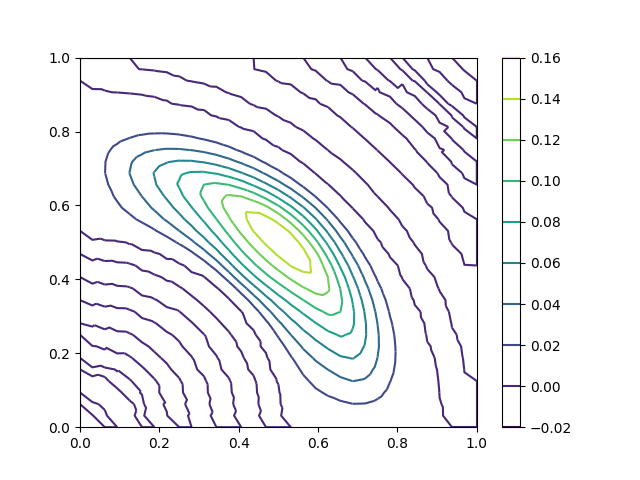}
    \caption{VP, $k=2$}
    \label{l2forcevpplot2}
  \end{subfigure}
  \begin{subfigure}[t]{0.32\textwidth}
    \includegraphics[width=\textwidth]{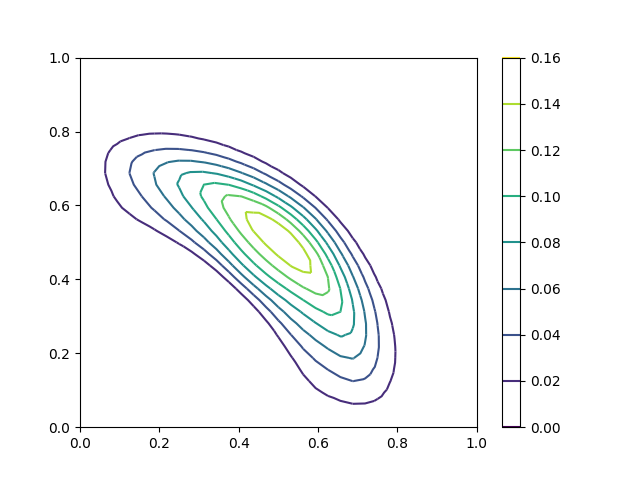}
    \caption{VI, $k=2$}
    \label{l2forceviplot2}
  \end{subfigure}
  \begin{subfigure}[t]{0.32\textwidth}
    \includegraphics[width=\textwidth]{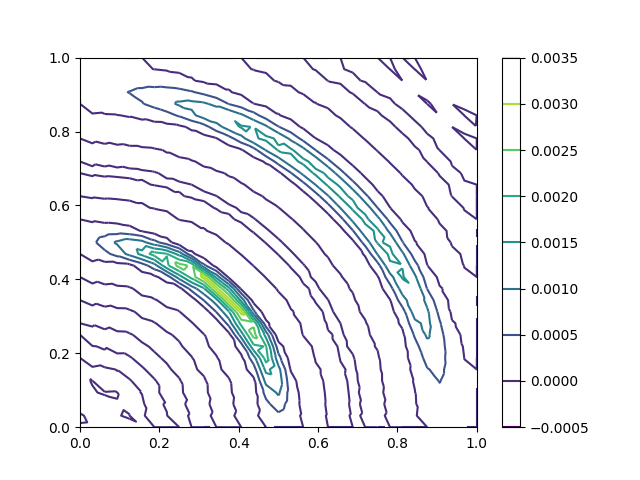}
    \caption{Difference, $k=2$}
    \label{l2forcediffplot2} 
  \end{subfigure} \\
  \begin{subfigure}[t]{0.32\textwidth}
    \includegraphics[width=\textwidth]{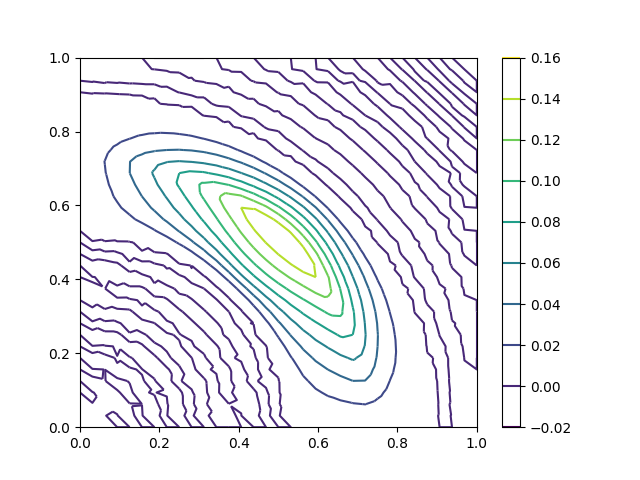}
    \caption{VP, $k=3$}
    \label{l2forcevpplot3} 
  \end{subfigure}
  \begin{subfigure}[t]{0.32\textwidth}
    \includegraphics[width=\textwidth]{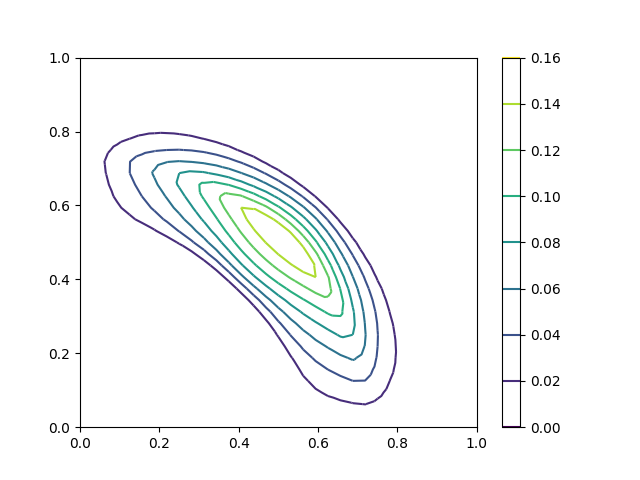}
    \caption{VI, $k=3$}
    \label{l2forceviplot3}
  \end{subfigure}
  \begin{subfigure}[t]{0.32\textwidth}
    \includegraphics[width=\textwidth]{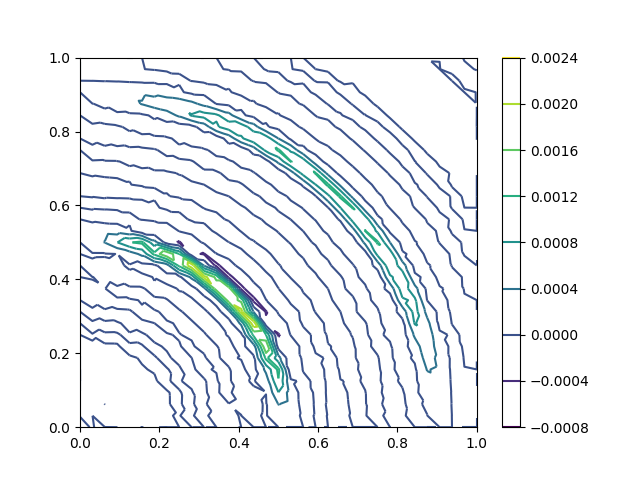}
    \caption{Difference, $k=3$}
    \label{l2forcediffplot3} 
  \end{subfigure} \\  
  \caption{Solutions of the variational problem (VP), variational inequality (VI), and their difference using linear, quadratic, and cubic basis functions with forcing function~\eqref{eq:roughf}.}
  \label{fig:l2forceplot}
\end{figure}

Figure~\ref{fig:l2forceplot} plots the solutions obtained from solving both the variational problem and inequality with linear, quadratic, and cubic Bernstein polynomials.
In each case, we see oscillations producing negative values in the solutions to the variational problem, while these are absent from the solutions to the variational inequalities.
Also, we see some increase in resolution with higher-degree polynomials.  

\subsection{Stationary convection-diffusion}
We now turn to the convection-diffusion equation~\eqref{eq:convdiff} and its SUPG formulation~\eqref{eq:supg}.
To demonstrate the effect of higher-order discretization in both a variational problem and bounds-constrained variational inequality, we take the two-dimensional benchmark given in~\cite{chang2017variational}.
Here, the domain $\Omega$ is the unit square with the square $[\tfrac{4}{9}, \tfrac{5}{9}]^2$ removed from the center.
The convective velocity is given by the spatially varying field
\begin{equation}
  \beta = (\cos \pi y^2, \sin 2 \pi x + \cos 2 \pi x^2),
\end{equation}
and $\kappa$ is given by the diffusion/dispersion tensor
\begin{equation}
  \kappa = \left( \alpha_T \| \beta \| + D_M \right) I + \left( \alpha_L - \alpha_T \right) \frac{\beta \otimes \beta}{\| \beta \|}.
\end{equation}
Here, $\alpha_L$ and $\alpha_T$ denote the longitudinal and transverse dispersivity, respectively, and $D_m$ is the molecular diffusivity.
We set these with values
\begin{equation}
  \alpha_L = 10^{-1}, \ \ \ \alpha_T = 10^{-5}, \ \ \ D_m = 10^{-9}.
\end{equation}
We apply Dirichlet boundary conditions, with $u=0$ on the exterior and $u=1$ on the boundary of the deleted square.

Figure~\ref{fig:supgcoarse} shows the solution using the linear variational problem and variational inequalities on a relatively coarse mesh consisting of 5,440 triangles and 2,832 vertices.
With the variational problem, we see widespread oscillations giving negative solutions in a large portion of the domain, and solutions also rise above 1 near the inner boundary.
The features are sharper, but oscillations worse, with quadratic elements than linear ones.
The solution to the variational inequality maintains this comparable resolution while removing the oscillatory behavior and bounds violations.
Figures~\ref{supg1diff} and~\ref{supg2diff} show the differences between the solutions to the variational problem and inequality.
Unfortunately, the PETSc variational inequality solver did not converge with cubic basis functions, so at this time we are unable to report on higher-order approximations.

\begin{figure}
  \centering
  \begin{subfigure}[t]{0.30\textwidth}
    \includegraphics[width=\textwidth]{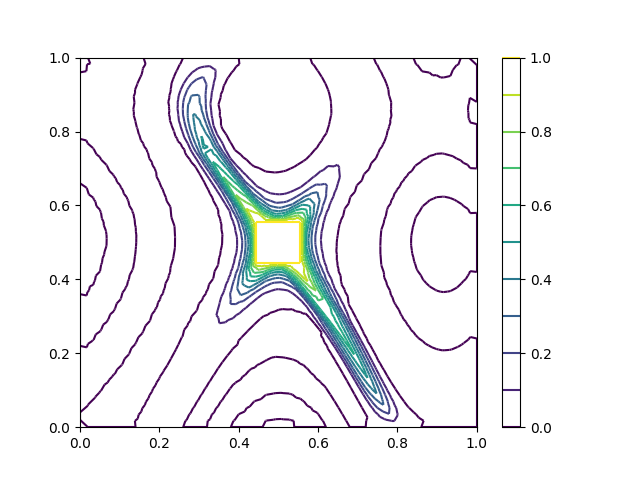}
    \caption{VP, $k=1$}
    \label{supgvp1}    
  \end{subfigure}
  \hfill
  \begin{subfigure}[t]{0.30\textwidth}
    \includegraphics[width=\textwidth]{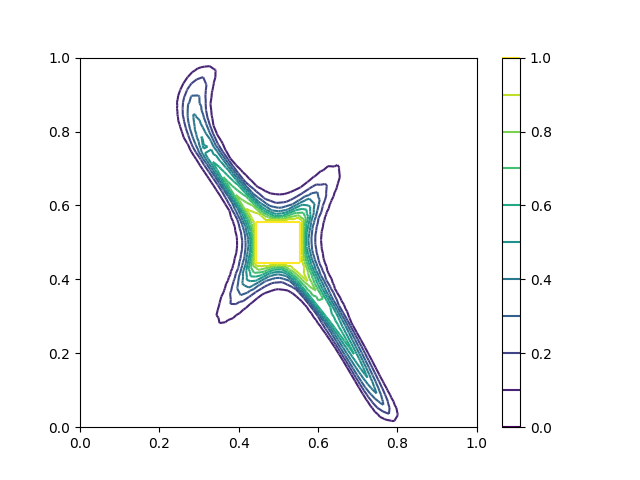}
    \caption{VI, $k=1$}
    \label{supgvi1}
  \end{subfigure}
  \hfill
  \begin{subfigure}[t]{0.30\textwidth}
    \includegraphics[width=\textwidth]{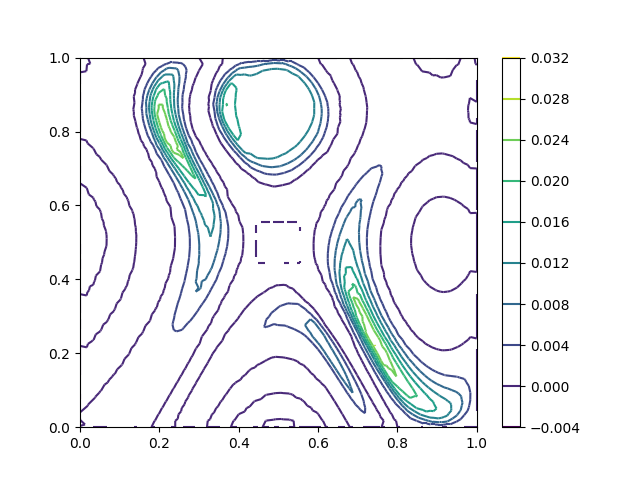}
    \caption{Diff, $k=1$}
    \label{supg1diff}
  \end{subfigure} \\  
  \begin{subfigure}[t]{0.30\textwidth}
    \includegraphics[width=\textwidth]{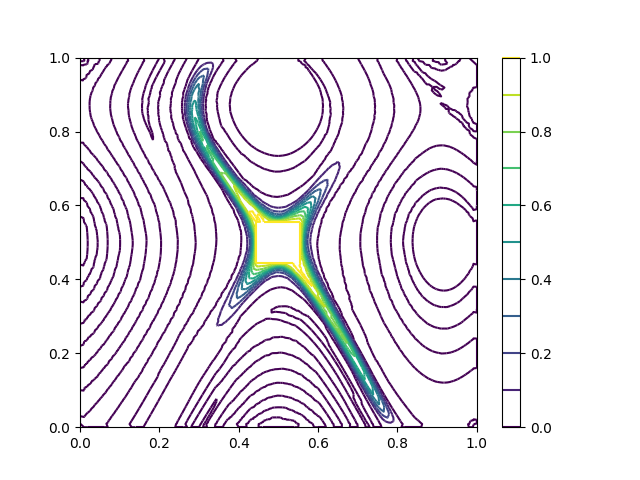}
    \caption{VP, $k=2$}
    \label{supgvp2}
  \end{subfigure}
  \hfill
  \begin{subfigure}[t]{0.30\textwidth}
    \includegraphics[width=\textwidth]{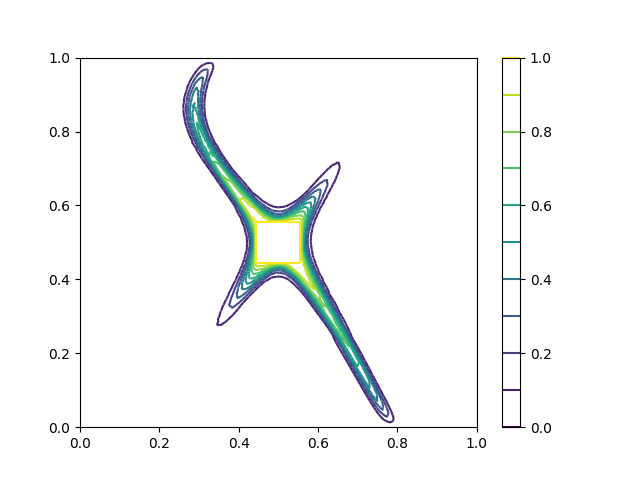}
    \caption{VI, $k=2$}
    \label{supgvi2}
  \end{subfigure}
  \hfill
  \begin{subfigure}[t]{0.30\textwidth}
    \includegraphics[width=\textwidth]{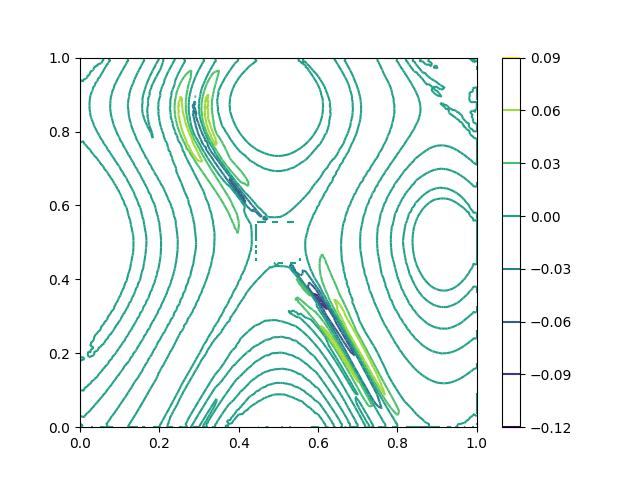}
    \caption{Diff, $k=2$}
    \label{supg2diff}
  \end{subfigure} \\    
  \caption{Solutions of the convection-diffusion problem using the SUPG finite element discretization (VP) and variational inequality (VI) with linear and quadratic Bernstein approximations on a coarse mesh of 5,440 triangles and 2,832 vertices.}
  \label{fig:supgcoarse}
\end{figure}

We repeated this experiment on a once-refined mesh with 21,760 triangles and 11,104 vertices, showing the result in Figure~\ref{fig:supgfine}, with the variational inequality again eliminating the bounds violations present in the variational problem.
The results are similar, but more sharply resolved, than on coarser mesh.
We note that quadratic elements on the coarse mesh also seem to give slightly sharper resolution than linears on the finer mesh.

\begin{figure}
  \centering
  \begin{subfigure}[t]{0.32\textwidth}
    \includegraphics[width=\textwidth]{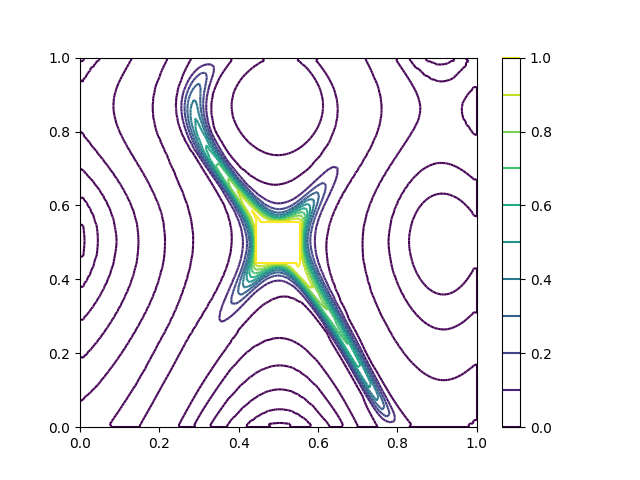}
    \caption{VP, $k=1$}
    \label{supgvp1fine}
  \end{subfigure}
  \begin{subfigure}[t]{0.32\textwidth}
    \includegraphics[width=\textwidth]{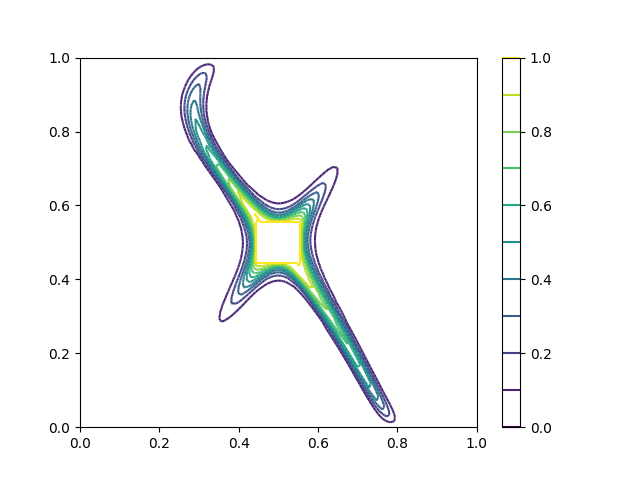}
    \caption{VI, $k=1$}
    \label{supgvi1fine}
  \end{subfigure}
  \begin{subfigure}[t]{0.32\textwidth}
    \includegraphics[width=\textwidth]{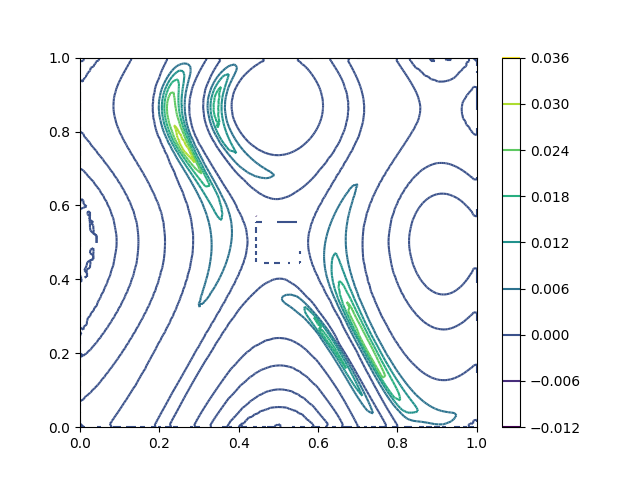}
    \caption{Diff, $k=1$}
    \label{supg1difffine}
  \end{subfigure} \\
  \begin{subfigure}[t]{0.32\textwidth}
    \includegraphics[width=\textwidth]{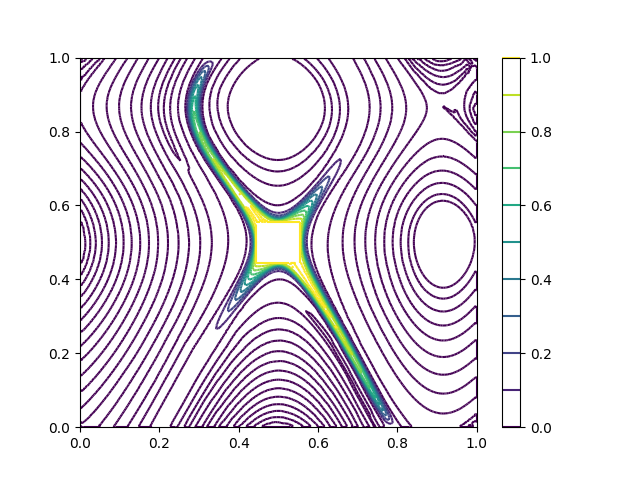}
    \caption{VP, $k=2$}
    \label{supgvp2fine}
  \end{subfigure}
  \begin{subfigure}[t]{0.32\textwidth}
    \includegraphics[width=\textwidth]{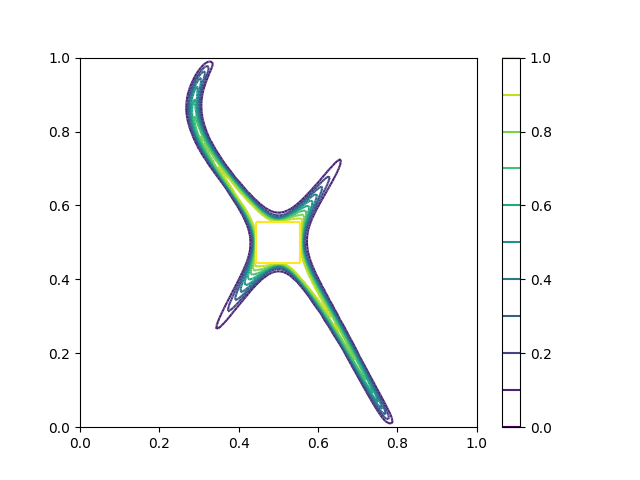}
    \caption{VI, $k=2$}
    \label{supgvi2fine}
  \end{subfigure}
  \begin{subfigure}[t]{0.32\textwidth}
    \includegraphics[width=\textwidth]{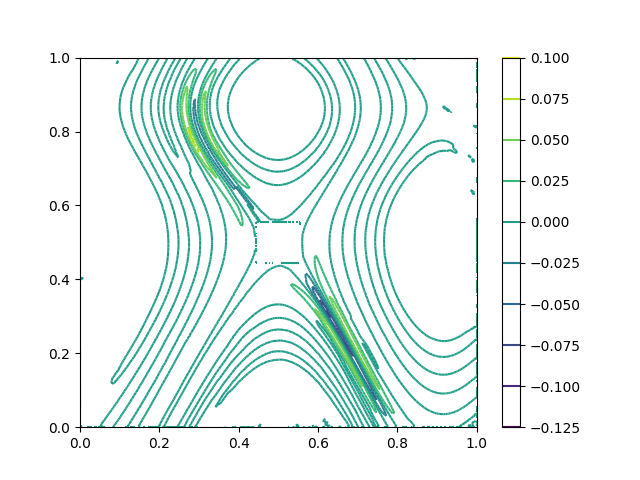}
    \caption{Diff, $k=2$}
    \label{supg2difffine}
  \end{subfigure}
  \caption{Solutions of the convection-diffusion problem using the SUPG finite element discretization (VP) and variational inequality (VI) with linear and quadratic Bernstein approximations on a finer mesh of 21,760 triangles and 11,104 vertices.}
  \label{fig:supgfine}
\end{figure}

\subsection{Time-dependent convection-diffusion}
Now, we return to the time-dependent convection-diffusion problem~\eqref{eq:tdcdgal}, where we apply our techniques to a rotating cone problem.
We take $\Omega = \left[-\tfrac{1}{2}, \tfrac{1}{2}\right]^2$, divided into a $ N \times N$ mesh of squares, each divided into two right triangles.
The diffusion coefficient is taken to be the scalar constant $\kappa = 10^{-4}$, and the convective velocity is given by the spatially varying divergence-free field
\begin{equation}
  \beta = \begin{bmatrix} -y \\ x \end{bmatrix},
\end{equation}
which creates a circular rigid rotation with period $2\pi$.
The initial condition for the problem is taken as a cone with of radius $0.1$ and height 1 centered at the point $(0.225, 0)$.

We discretize~\eqref{eq:tdcdgal} with the implicit midpoint rather than backward Euler method, giving a less diffusive and (formally) higher order method in time.  We take time step $\tau = \tfrac{1}{N}$.
The variational problem obtained at each time step is then approximated alternatively with a Galerkin finite element or discrete variational inequality.
In each case, we use continuous piecewise polynomials of degrees 1, 2, and 3, expressed in the Bernstein basis.

For the initial condition, we find the closest member of $V_h$ to the initial cone subject to the bounds constraints (i.e.~we solve a variational inequality), which is shown in Figure~\ref{fig:ic}.
Figure~\ref{fig:tdcdresults} shows the solutions obtained from the standard Galerkin method and the bounds-constrained variational inequality using various orders of spatial discretization.
We see that the variational problems, though stable, give significant violation of the bounds constraints, whild the variational inequality leads to very good solutions.
Increasing the polynomial degree leads to a sharper, more circular solution.

\begin{figure}
  \begin{center}
    \includegraphics[width=0.5\textwidth]{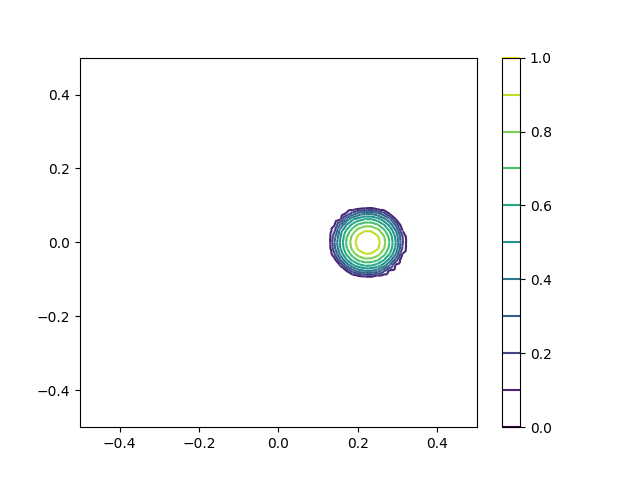}
  \end{center}
  \caption{Initial condition for rotating cone problem, approximated with cubic Bernstein polynomials subject to bounds constraints}
  \label{fig:ic}
\end{figure}

\begin{figure}
  \centering
  \begin{subfigure}[t]{0.30\textwidth}
    \includegraphics[width=\textwidth]{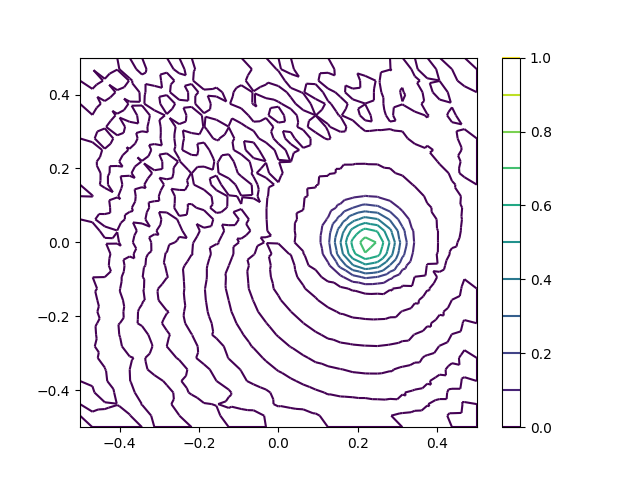}
    \caption{VP, $k=1$}
    \label{vpk1}    
  \end{subfigure}
  \hfill
  \begin{subfigure}[t]{0.30\textwidth}
    \includegraphics[width=\textwidth]{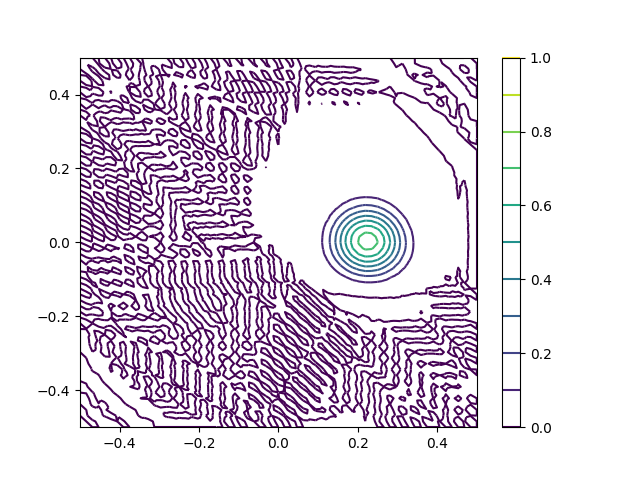}
    \caption{VP, $k=2$}
    \label{vpk2}
  \end{subfigure}
  \hfill
  \begin{subfigure}[t]{0.30\textwidth}
    \includegraphics[width=\textwidth]{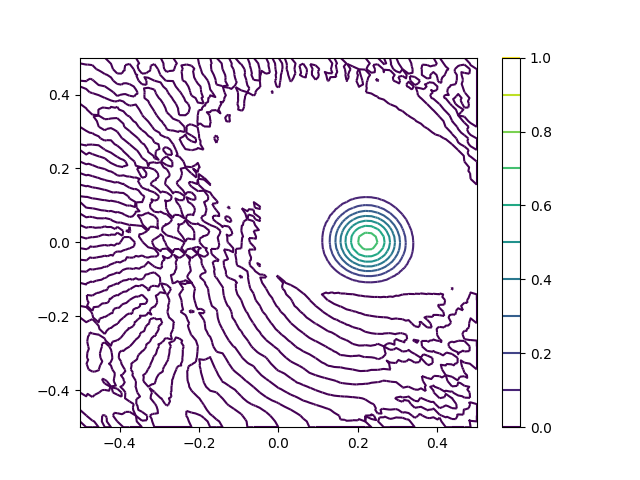}
    \caption{VP, $k=3$}
    \label{vpk3}
  \end{subfigure} \\
  \begin{subfigure}[t]{0.30\textwidth}
    \includegraphics[width=\textwidth]{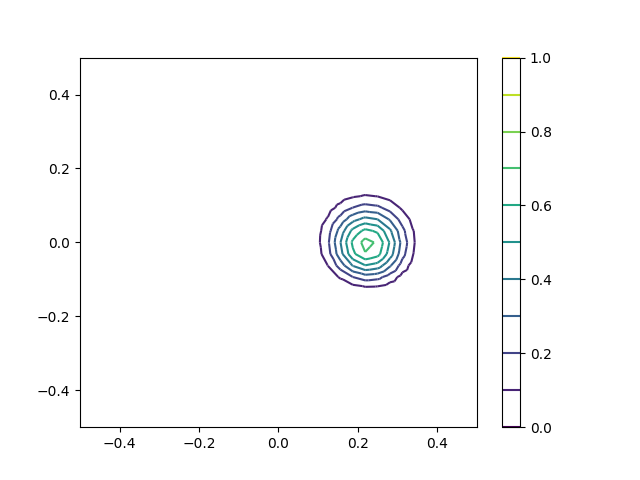}
    \caption{VI, $k=1$}
    \label{vik1}    
  \end{subfigure}
  \hfill
  \begin{subfigure}[t]{0.30\textwidth}
    \includegraphics[width=\textwidth]{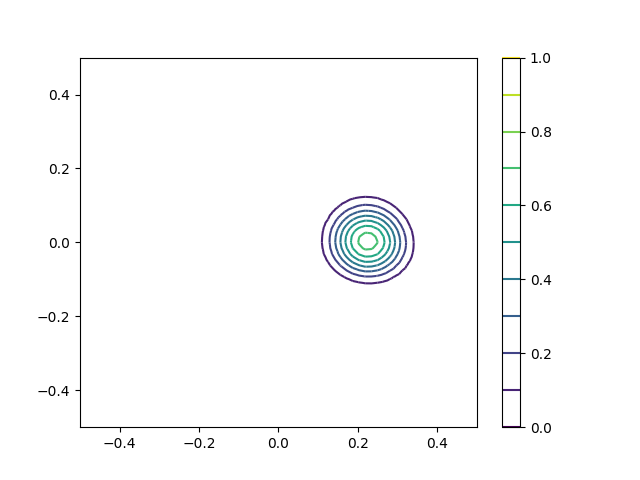}
    \caption{VI, $k=2$}
    \label{vik2}
  \end{subfigure}
  \hfill
  \begin{subfigure}[t]{0.30\textwidth}
    \includegraphics[width=\textwidth]{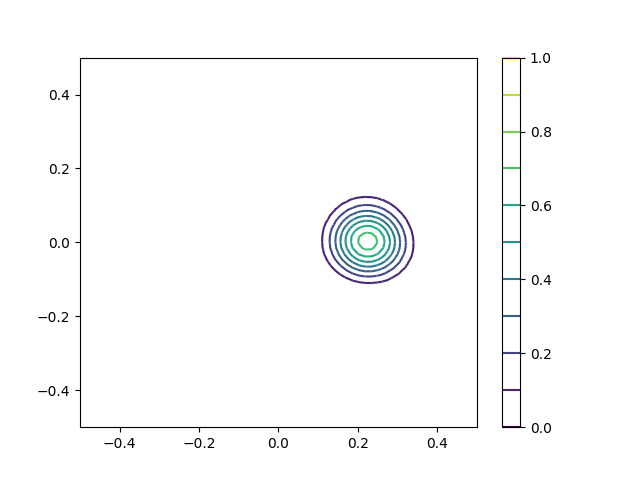}
    \caption{VI, $k=3$}
    \label{vik3}
  \end{subfigure}
  \caption{Rotating cone solutions after one rotation, comparing various orders of approximation to a discrete variational problem (VP) and variational inequality (VI).}
  \label{fig:tdcdresults}
\end{figure}

\section{Conclusions}
\label{sec:conc}
Linear variational inequalities can be used to enforce bounds constraints for partial differential equations discretized over finite element spaces.
In this paper, error estimates justify this approach, showing that the error obtained is comparable to the best approximation of the PDE solution in the set of discrete, bounds-respecting functions.
Moreover, restricted range approximation of smooth functions is shown to have comparable accuracy to unconstrained approximation.
Through the Bernstein basis, we are able to work with bounds-constrained polynomials.
Although applying the constraints to Bernstein coefficients does not produce all bounds-constrained polynomials, practical results indicate that this may be sufficient to produce genuinely higher-order approximations.

This paper leaves many directions for future research.
For one, our theoretical results still leave open the issue of characterizing the approximation power of constructing $U_h$ by imposing bounds constraints on the Bernstein control net rather than using all bounds-constrained members of the finite element space. 
Also, finding efficient solvers for higher-order discretizations of variational inequalities is a challenging problem.
We have only used one black-box algorithm, and have not studied other approaches~\cite{bueler2023full,de1996semismooth,hoppe1987multigrid}.
Further, the literature suggests the variational inequality framework combines well with more general discretizations and operators, but we have yet to address this from a theoretical standpoint.
Finally, variational inequalities can be formulated for time-dependent problems, but finding higher-order discretizations in both space and time seems to be an open issue.

\bibliographystyle{plain}
\bibliography{paper}

\end{document}